\documentclass[12pt,draft]{amsart}
\usepackage[utf8]{inputenc}
\usepackage{amsmath,amssymb,amsbsy,amsfonts,amsthm,latexsym,mathabx,
            amsopn,amstext,amsxtra,euscript,amscd,stmaryrd,mathrsfs,
            cite,array,mathtools,enumerate}

\usepackage{url}
\usepackage{a4wide}

\newtheorem{theorem}{Theorem}
\newtheorem{lemma}[theorem]{Lemma}

\newcommand{\F}{\mathbb{F}}
\newcommand{\K}{\mathbb{K}}
\newcommand{\ZZ}{\mathbb{Z}}
\newcommand{\Fq}{\mathbb{F}_q}
\newcommand{\e}{\mathbf{e}}
\newcommand{\mand}{ \quad \text{and} \quad}

 \usepackage{todonotes}

\def\cA{{\mathcal A}}
\def\cB{{\mathcal B}}

\def\cI{{\mathcal I}}

\def\cM{{\mathcal M}}

\def\cR{{\mathcal R}}
\def\cS{{\mathcal S}}
\def\cT{{\mathcal T}}
\def\cU{{\mathcal U}}

\def\cW{{\mathcal W}}

\newcommand{\D}{\operatorname{\mathrm{D}}}

\usepackage{soul}

\title{On divisors of sums of polynomials}
 \author[L.~M{\'e}rai]{L{\'a}szl{\'o} M{\'e}rai}
 \address{Johann Radon Institute for Computational and Applied Mathematics, Austrian Academy of Sciences,  Altenberger Stra\ss e 69, A-4040 Linz, Austria} 
 \email{laszlo.merai@oeaw.ac.at}

\allowdisplaybreaks


\date{\today}
\begin{document}

\maketitle
\begin{abstract}
Let $\cA$ and $\cB$ be sets of polynomials of degree $n$ over a finite field. We show, that if $\cA$ and $\cB$ are large enough, then $A+B$ has an irreducible divisor of large degree for some $A\in\cA$ and $B\in \cB$.  
\end{abstract}

\section{Introduction}

For an integer $k$, let $P(k)$ denote the largest prime divisor of $k$ with the convention $P(0)=0$ and $P(\pm 1)=1$.

For a given set $\cA\subset \{1,\dots, n\}$ of integers,
it is a classical number theoretic question to study 
$$
\{P(a): a\in \cA \},
$$
see \cite{cite:BakerHarman,cite:Balog14,cite:LucaShp05,cite:Ma,cite:SarkozySt86,cite:SarkozySt86b,cite:SarkozySt00,cite:ShpSu07,cite:St05,cite:St01,cite:St04}.

For example, Sárközy and Stewart \cite{cite:SarkozySt86,cite:SarkozySt86b} studied the prime divisors of the elements of sum-sets $\cA+\cB$. They showed that if $\cA,\cB$ are not too small, then there are $a\in\cA$ and $b\in\cB$ such that $P(a+b)$ is large.  
In particular, they showed, that if $\cA,\cB\subset \{1,\dots, n\}$ have positive relative density, i.e. $\#\cA, \#\cB \geq c_1 n$ for some $c_1>0$, then 
$$
\max_{a\in\cA,b\in\cB} P(a+b) \geq c_2 n,
$$
for some positive constant $c_2>0$ which may depend only on $c_1$.

In this paper we investigate this problem for polynomials over finite fields. More precisely, for a prime power $q$, let $\Fq$ denote the finite field of $q$ elements. For a polynomial $A\in \Fq[T]$, let $\D(A)$ denote the maximal degree of irreducible divisors of $A$, that is,
$$
\D(A)=\max_{\substack{P\mid A\\ P \text{ irreducible}}} \deg P,
$$
with the convention, that $\D(a)=0$ for all $a\in\Fq$.

Let $\log$ denote the natural logarithm and for $a\geq 2$ put  $\log_a(x)=\max\{\log(x)/\log(a); 1 \}$.

For a positive integer $n$, let $\cM_n$ denote the set of \emph{monic} polynomials over $\Fq$ of degree~$n$.

We will prove the following result.

\begin{theorem}\label{thm:main}
Let $\alpha, \beta \in \Fq^*$ with
$$
\alpha + \beta =1.
$$

For any $\varepsilon>0$, $n>n_0(\varepsilon)$ and $\cA,\cB \subset \cM_n$ such that
\begin{equation}\label{eq:AB_bound}
\left(\#\cA \#\cB\right)^{1/2} \geq q^{(6/7+\varepsilon)n},
\end{equation}
there exist at least $c_1 \#\cA \#\cB /n $ pairs $(A,B)\in\cA\times\cB$ such that
$$
\D(\alpha A+\beta B)\geq 
n-\log_q \rho - \log _q\log_q \log_q \rho -\frac{c_2}{\log q} -1
$$
for some $c_1,c_2>0$ which  may depend on $\varepsilon$,
where
$$
\rho = \frac{q^n}{\left(\#\cA \#\cB\right)^{1/2} }.
$$

\end{theorem}

In particular, if $\#\cA, \#\cB\geq c q^n$ for some $c>0$, then there are polynomials $A\in\cA$, $B\in\cB$ such that 
$$
\D(\alpha A+\beta B)\geq n-c'
$$
for come constant $c'>0$ which may depend only on $c$.

\section{Notations}

For given functions $F$ and $G$, the notations $F\ll G$, $G \gg F$ and $F =O(G)$ are all equivalent to the statement that the inequality $|F| \leq c|G|$
holds with some constant $c > 0$. Throughout the paper, any implied
constants in symbols O, $\ll$ and $\gg$ are absolute unless specified otherwise.
\smallskip

For positive integer $n$ we denote by $\cI_n$ the set of \emph{monic irreducible}  polynomials of degree~$n$. We also let $\cM=\bigcup_n \cM_n$ and $\cI=\bigcup_n \cI_n$ to be the sets of \emph{all monic} and \emph{all monic irreducible polynomials} respectively.

Let $\K_{\infty}$ be the set of formal power series 
$$
\Fq((1/T))=\left\{\xi=\sum_{i\leq k}x_iT^i, k\in \ZZ, x_i\in\Fq \right\},
$$ 
which is the completion of $\Fq[T]$ with respect to the usual norm
$$
|A|=q^{\deg A}
$$
for polynomial $A$ (with the convention $|0|=0$). We extend this norm to $\K_{\infty}$ by
$$
|\xi|=q^k
$$
where $k$ is the largest index so that $x_k\neq 0$.

Moreover, for $\xi\in \K_{\infty}$ we define
$$
\| \xi\|=\min_{A\in \Fq[T] }|\xi-A|.
$$

Write
$$
\psi(x)=\exp\left(\frac{2\pi i }{p} \mathrm{tr}_{\Fq}(x) \right), \quad x\in\F_q,
$$
where $\mathrm{tr}_{\Fq}$ is the (absolute) trace of $\Fq$ and $p$ is the characteristic of $\Fq$, and put
$$
\e(\xi)=\psi(x_{-1}), \quad \xi=\sum_{i\leq k}x_iT^i \in  \K_{\infty} .
$$

We define $\mathbf{T}$ by $\{\xi\in \K_{\infty}: \|\xi\|<1 \}$, and fix an additive Haar measure, normalized so that $\int_{\mathbf{T}} \mathrm{d}\xi=1$. Then we have for a polynomial $A\in\Fq[T]$,
\begin{equation}\label{eq:ortogonal}
 \int_{\mathbf{T}} \e(A\xi) \mathrm{d}\xi=
 \left\{
 \begin{array}{cl}
    1  & \text{if } A=0, \\
    0  &  \text{if } A\neq 0,
 \end{array}
 \right.
\end{equation}
see \cite[Theorem~3.5]{Hayes:Goldbach}.

\section{Outline of the proof}

Write
$$
f_\cA(\xi)=\sum_{ A\in\cA}\e(\alpha A\xi)\mand
f_\cB(\xi)=\sum_{B\in\cB}\e(\beta B\xi).
$$
Then
$$
f_\cA(\xi)f_\cB(\xi)=\sum_{A\in\cA, B\in\cB}\e((\alpha A+\beta B)\xi)=\sum_{G\in \cM_n} u_G\e(G\xi ),
$$
where
\begin{equation}\label{eq:u_def}
u_G=\sum_{\substack{\alpha A+\beta B=G\\ A\in\cA, B\in\cB}}1
\end{equation}
counts the number of representations of $G$ as a sum of $G=\alpha A+\beta B$, where $A\in\cA$ and  $B\in\cB$.

Put
\begin{equation*}
j=\left\lceil \log_q \rho + \log_q\log_q \log_q \rho  +\frac{\eta}{\log q}\right\rceil
\end{equation*}
for some $\eta>0$ to be fixed later. 

Write
$$
\cS=\{ C P: C\in \cM_{j}, P\in\cI_{n-j} \}.
$$
and define
$$
f_\cS(\xi)=\sum_{S\in\cS}\e(S\xi)=\sum_{G\in\cM_n}v_G\e(G\xi),
$$
where
$$
v_G
=
\left\{
\begin{array}{cl}
1     & \text{if } G=CP \text{ for some }  C\in \cM_{j}, P\in\cI_{n-j} , \\
0     & \text{otherwise}.
\end{array}
\right.
$$

Finally, define the following integral
$$
I=\int_{\mathbf{T}}f_\cA(\xi)f_\cB(\xi)f_\cS(-\xi)\mathrm{d}\xi.
$$
By \eqref{eq:ortogonal}, we have
\begin{align*}
I&=\int_{\mathbf{T}}\sum_{A\in\cA}\sum_{B\in\cB}\sum_{S\in\cS} \e((\alpha A+\beta B-S)\xi) \mathrm{d}\xi\\
&=\int_{\mathbf{T}}\sum_{G\in\cM_n } \sum_{H\in\cM_n} u_G v_H \e((G-H)\xi) \mathrm{d}\xi\\
&=\sum_{G\in\cM_n}  u_G v_G.
\end{align*}

Note that $u_G=1$ implies that $G$ has the form $G=\alpha A+\beta B$ for some $A\in\cA$, $B\in\cB$ and $v_G=1$ implies 
that $G$ has an irreducible divisor of degree $j$. Thus, in order to prove Theorem~\ref{thm:main}, it is enough to show 
\begin{equation}\label{eq:goal}
I\geq c_1 \frac{\#\cA\#\cB}{n}
\end{equation}
which is shown in Section~\ref{sec:proof}.

\section{Preliminaries}

\subsection{Diophantine approximation for polynomials}
The following result \cite[Theorem~4.3]{Hayes:Goldbach} is the polynomial analogue of the Dirichlet's theorem on Diophantine approximation of real numbers.

\begin{lemma}\label{lemma:dio}
For $\xi\in\mathbf{T}$ and positive integer $n$, there exist unique $G,H\in\Fq[T]$, $\gcd(G,H)=1$ such that $H$ is monic, $|G|<|H|<q^{n/2}$ and
$$
\left|\xi -\frac{G}{H} \right|<\frac{1}{|H|q^{n/2}}.
$$
\end{lemma}

\subsection{Equivalence relation of polynomials}
For a monic polynomial 
$$
A=T^n+a_{n-1}T^{n-1}+\dots+a_0\in\Fq[T],
$$
we call $a_{n-1}, \dots, a_{n-s}$ the first $s$ coefficients of $A$ with the convention $a_{n-s}=0$ for $s>n$.  

For a non-negative integer $\ell$  and a  polynomial $H$ write
$$
A\equiv B \mod \cR_{\ell,H}
$$
if $A \equiv B \pmod H$ and the first $\ell$ coefficients of $A$ and $B$ are the same. Clearly, $\cR_{\ell,H}$ is an equivalence relation. We also write $\cR_{\ell}=\cR_{\ell,1}$ and $\cR_{H}=\cR_{0,H}$.

The polynomial $A$ is invertible modulo $\cR_{\ell,H}$ if $\gcd(A,H)=1$. These invertible polynomials form a group, denoted by $G_{\ell, H}=\left(\cM/\cR_{\ell,H}\right)^\times$. Write $G_{\ell}=\left(\cM/\cR_{\ell}\right)^\times$ and $G_{H}=\left(\cM/\cR_{H}\right)^\times$. Then we have 
\begin{equation}\label{eq:CRT}
G_{\ell, H} \cong G_{\ell} \times G_{H},
\end{equation}
see \cite[Theorem~8.6]{Hayes:Distribution} and \cite[Lemma~1.1]{Hsu:Distribution}. We also have $\# G_{\ell}=q^\ell$ and
$$
\#G_{H}=\#(\Fq[T]/(H))=\Phi(H).
$$
We have the following lower bound on $\Phi(H)$, see \cite[Lemma~2.3]{Ha:Irreducible_polynomials},
\begin{equation}\label{eq:phi}
\Phi(H) \gg 
\frac{|H|}{\log_q (\deg H) +1 } \quad \text{for } H\neq 0.
\end{equation}

\subsection{Distribution of irreducible polynomials in residue classes}

We have the following result on distribution of irreducible polynomials in residue classes modulo $\cR_{\ell,H}$, see \cite[Corollary~2.5]{Hsu:Distribution}.

\begin{lemma}\label{lemma:combined_PNT-2}

Let $H\in \Fq[T]$ and $\ell$ be a non-negative integer. For $B\in\Fq[T]$ with $\gcd(B,H)=1$, we have
$$
\#\left\{P\in \cI_n: P \equiv B \mod \cR_{\ell, H}\right\}
 = \frac{1}{q^\ell \Phi(H)}\frac{q^n}{n} + O\left(\frac{(\ell + \deg H) q^{n/2}}{n}\frac{|H|}{\Phi(H)}\right).
$$
\end{lemma}

\subsection{Exponential sums over polynomials}

For positive integers $k$, we have
\begin{equation}\label{eq:sum}
\sum_{A\in\cM_k}\e(A\xi)=
\left\{
\begin{array}{ll}
     q^k \e(T^k\xi) & \text{if } \|\xi\|<q^{-k},  \\
     0 & \text{otherwise,} 
\end{array}
\right.
\end{equation}
see \cite[Theorem~3.7]{Hayes:Goldbach}.

\begin{lemma}
Let 
\begin{equation}\label{eq:Va_00}
\xi= \frac{A}{B}+ \gamma \quad \text{with} \quad |\gamma| \leq \frac{1}{|B|^2} \quad \text{and} \quad (A,B)=1.
\end{equation}
Then we have
\begin{equation}\label{eq:Va_03}
    \sum_{G\in\cM_\ell}\left|\sum_{H \in \cM_k} \e(\xi GH) \right|\ll \frac{q^{k+\ell}}{|B|}  + q^\ell +|B|.
\end{equation}

\end{lemma}
\begin{proof}
We have by \eqref{eq:sum}, that
$$
\left|\sum_{H \in\cM_k} \e(\xi GH) \right|=
\left\{
\begin{array}{cl}
     q^k & \text{if } \|\xi G\|<q^{-k} ,\\
     0 & \text{otherwise}. 
\end{array}
\right.
$$
First assume, that $\deg B\leq \ell$. Then write $G=CB+D$ with $\deg D<\deg B$. Then we have 
$$
\left\|\xi G\right\|=\left\|\frac{AD}{B} + \gamma D + \gamma CB \right\|.
$$

If $D\neq 0$, we have
\begin{equation}\label{eq:D_small}
\left\|\frac{AD}{B}\right\|\geq \frac{1}{|B|} \mand \|\gamma D\|< \frac{1}{|B|}
\end{equation}
by \eqref{eq:Va_00}
so
\begin{align*}
&\#\left\{1\leq |D|<|B|: \left\|\frac{AD}{B}+\gamma D + \gamma CB\right\|<\frac{1}{q^k} \right\}\\
\leq   &
\#\left\{1\leq |D|<|B|: \left\|\frac{AD}{B}+\gamma D + \gamma CB\right\|<\max\left\{\frac{1}{q^k}, \frac{1}{|B|}\right\} \right\}\\
\leq   &
\#\left\{1\leq |D|<|B|: \left\|\frac{AD}{B} + \gamma CB\right\|<\max\left\{\frac{1}{q^k}, \frac{1}{|B|}\right\} \right\}\\
\leq   &
\max\left\{1,\frac{|B|}{q^k}\right\}.
\end{align*}

Then we have
\begin{align*}
&     \sum_{|C|=q^\ell/|B|}q^k\sum_{\deg D<\deg B} \mathbf{1} \left\{ \left\|\frac{AD}{B} + \gamma D + \gamma CB \right\|< \frac{1}{q^k}\right\}
\ll  \sum_{|C|=q^\ell /|B|}q^k \left(1 + \frac{|B|}{q^k} \right)
\leq  \frac{q^{\ell + k}}{|B|} + q^\ell. 
\end{align*}

For $\deg B>\ell$, we have $C=0$ and in the same way one gets by \eqref{eq:D_small} that
\begin{align*}
&\#\left\{ \deg D=\ell: \left\|\frac{AD}{B}+\gamma D \right\|<\frac{1}{q^k} \right\}\\
=&\#\left\{ \deg D=\ell: \left\|\frac{AD}{B} \right\|<\frac{1}{q^k} \right\}\\
\leq
&
\#\left\{1\leq |D|<|B|: \left\|\frac{AD}{B}\right\|<\frac{1}{q^k} \right\}
\leq \frac{|B|}{q^k}
\end{align*}
which yields
$$
q^k\sum_{ D\in \cM_\ell} \mathbf{1} \left\{ \left\|\frac{AD}{B} + \gamma D \right\|< \frac{1}{q^k}\right\} \leq |B| .
$$
\end{proof}

\subsection{Vaughan's estimate}
Write 
$$
\mu (F)=
\left\{
\begin{array}{cl}
(-1)^k & \text{if $F=P_1\dots P_k$ for some distinct $P_1,\dots, P_k\in\cI$},\\
0 & \text{otherwise}
\end{array}
\right.
$$
and
$$
\Lambda (F)=
\left\{
\begin{array}{cl}
\deg P & \text{if $F=P^k$ for some $P\in \cI$ and integer $k\geq 1$},\\
0 & \text{otherwise}.
\end{array}
\right.
$$
Clearly, we have
\begin{equation}\label{eq:Lambda=deg}
\sum_{G\mid F}\Lambda(G)=\deg F 
\end{equation}
and the analogue of the prime number theorem for polynomials states
\begin{equation}\label{eq:PNT}
    \sum_{F\in \cM_n}\Lambda(F)=q^n.
\end{equation}

We need the following analogue of Vaughan's identity whose proof is identical to \cite[Proposition~13.4]{IK}.

\begin{lemma}\label{lemma:Vaughan1}
Let $y,z> 1$. For $\deg F>z$,
\begin{align*}
    \Lambda(F) = &\sum_{\substack{G\mid F\\ \deg G\leq y}}\mu(G) \deg(F/G)- \mathop{\sum \sum}_{\substack{G H \mid F \\ \deg G\leq y, \deg H\leq z}} \mu (G) \Lambda(H)
    +
    \mathop{\sum \sum}_{\substack{G H \mid F \\ \deg G> y, \deg H> z}} \mu (G) \Lambda(H).
\end{align*}
\end{lemma}

\begin{lemma}\label{lemma:irred_sum}
Assume $\xi$ satisfies \eqref{eq:Va_00}. Then
\begin{equation*}
    \sum_{P\in\cI_n}\e(\xi P) \ll 
    n^{3/2}\left(\frac{q^n}{|B|^{1/2}} + q^{4n/5} + |B|^{1/2}q^{(n+1)/2}\right) .
\end{equation*}
\end{lemma}

\begin{proof}
We can assume, that $|B|<q^n$ otherwise the bound is trivial.

It follows from \eqref{eq:PNT} that
$$
\frac{1}{n}\sum_{\substack{k\mid n \\ k<n}}\sum_{F\in\cM_k}\Lambda(F)\leq q^{n/2}
$$
thus
\begin{equation}\label{eq:Vau_primepower}
     \sum_{P\in \cI_n}\e(\xi P)= \frac{1}{n}\sum_{P\in\cI_n}\Lambda(P)\e(\xi P)=
     \frac{1}{n}\sum_{F\in\cM_n}\Lambda(F)\e(\xi F)+O\left( q^{n/2}\right).
\end{equation}

By Lemma~\ref{lemma:Vaughan1}, we have
\begin{equation}\label{eq:Vau_split}
\begin{split}
\sum_{F\in\cM_n}\Lambda(F)\e(\xi F)
& = \mathop{\sum \sum}_{\substack{ GH\in\cM_n \\ \deg G\leq y   }} \mu(G) \deg (H) \e(\xi GH)\\
& \quad - \mathop{\sum \sum \sum}_{\substack{ GHI\in\cM_n \\ \deg G\leq y, \deg H\leq z}} \mu (G) \Lambda(H) \e(\xi GHI)\\
& \quad  +
    \mathop{\sum \sum \sum}_{\substack{GHI \in\cM_n \\ \deg G> y, \deg H> z}} \mu (G) \Lambda(H)\e(\xi GHI).
\end{split}
\end{equation}
We estimate the first term by~\eqref{eq:Va_03},
\begin{equation}\label{eq:Vau_sum1}
\begin{split}
\left|\mathop{\sum \sum}_{\substack{ GH\in\cM_n \\ \deg G\leq y   }} \mu(G) \deg (H) \e(\xi GH)\right|&\leq \sum_{\ell \leq y}  (n- \ell) \sum_{G\in\cM_\ell}\left|\sum_{H\in\cM_{n-\ell}} \e(\xi GH)\right| \\
& \leq \sum_{\ell \leq y} (n-\ell) \left(\frac{q^n}{|B|}  + q^\ell + |B|\right)\\
& \ll y n \left(\frac{q^n}{|B|} + q^y + |B|\right).
\end{split}
\end{equation}
Similarly, using \eqref{eq:Lambda=deg} we estimate the second term by
\begin{equation}\label{eq:Vau_sum2}
\begin{split}
&\left|    \mathop{\sum \sum \sum}_{\substack{GHI\in\cM_n \\ \deg G\leq y, \deg H\leq z}} \mu (G) \Lambda(H) \e(\xi GHI)\right|\\
\leq& \sum_{\deg J<y+z} \left(\mathop{\sum \sum}_{\substack{GH = J \\ \deg G\leq y, \deg H\leq z}} \Lambda(H)\right) \left|\sum_{|I|=q^n/|J|} \e(\xi J I)  \right|  \\
 \leq& (y+z) \sum_{\ell < y+z}\sum_{J\in\cM_\ell}  \left|\sum_{I\in\cM_{n-\ell}} \e(\xi J I)  \right|\\
 \ll & (y+z)^2 \left(\frac{ q^n}{|B|}+ q^{y+z} + |B| \right).
\end{split}
\end{equation}

For the last term, write $J=HI$ and note that by \eqref{eq:Lambda=deg} we have
$$
c(J)=\sum_{\substack{H \mid J \\ \deg H>{z}}}\Lambda(H)\leq \deg J.
$$
Then
\begin{align*}
&\left|\mathop{\sum \sum \sum}_{\substack{GHI \in \cM_n \\ \deg G> y, \deg H> z}} \mu (G) \Lambda(H)\e(\xi GHI)\right|
\leq \sum_{y <\ell<n}  \sum_{G\in\cM_\ell} \left|\sum_{J\in\cM_{n-\ell}}  c(J)\e(\xi GJ)\right|.
\end{align*}
For fixed $\ell$, we have by the Cauchy inequality that
\begin{align*}
 \left(\sum_{G\in \cM_\ell} \left|\sum_{J\in\cM_{n-\ell}}  c(J)\e(\xi GJ)\right|\right)^2
 &\leq q^\ell \sum_{G\in \cM_\ell} \left|\sum_{J\in\cM_{n-\ell}}  c(J)\e(\xi GJ)\right|^2\\
&\leq q^\ell (n-\ell)^2 \mathop{\sum \sum}_{\deg J_1=\deg J_2={n-\ell}} \left|\sum_{G\in \cM_\ell}\e(\xi G(J_1-J_2)) \right|. 
\end{align*}
Write $I_1-I_2=aI$ with monic $I$ of degree at most $n-\ell-1$ and $a\in\Fq^*$. Then
\begin{align*}
&\mathop{\sum \sum}_{\deg J_1=\deg J_2={n-\ell}} \left|\sum_{G\in \cM_\ell}\e(\xi G(J_1-J_2)) \right|\\
&=\sum_{a\in\Fq^*} \sum_{J_1\in \cM_{n-\ell}} \sum_{k<n-\ell} \sum_{I\in \cM_{k}}\left|\sum_{G\in \cM_\ell}\e(a\xi GI) \right|\\
& \leq q^{n-\ell+1}(n-\ell)\left( \frac{q^{n-1}}{|B|} + q^{n-\ell-1} + |B|\right).
\end{align*}

Thus, we get
\begin{equation}\label{eq:Vau_sum3}
\left|\mathop{\sum \sum \sum}_{\substack{GHI \in\cM_n \\ \deg G> y, \deg H> z}} \mu (G) \Lambda(H)\e(\xi GHI)\right|\ll q^{n/2}n^{5/2}\left(\frac{q^{n/2}}{\sqrt{|B|}} + \frac{q^{n/2}}{q^{y/2}} + \sqrt{q|B|} \right).
\end{equation}
Combining \eqref{eq:Vau_primepower} and \eqref{eq:Vau_split} with \eqref{eq:Vau_sum1}, \eqref{eq:Vau_sum2} and \eqref{eq:Vau_sum3} we get the result with the choice $y=z=2n/5$.
\end{proof}

\section{Further preliminaries}

The following result is the main tool to prove Theorem~\ref{thm:main}.

\begin{lemma}\label{lemma:main}
Assume, $1\leq j<n/7-4 \log_q n-3/7$ 
and 
\begin{equation}\label{eq:k}
k\leq n-j/2    .
\end{equation}
Then for  $|\xi|\geq q^{-k}$ we have
\begin{align*}
f_\cS(\xi)\ll&
\frac{\log j }{n}   q^{n-j}.
\end{align*}
\end{lemma}
\begin{proof}

Put
\begin{equation}\label{eq:Q}
Q=  q^{n -3j/2},\qquad  \omega = 2j + 5 \log_q n  \mand \cT=\{\xi\in \mathbf{T}: |\xi|\geq q^{-k} \}.
\end{equation}

By Lemma~\ref{lemma:dio}, for all $\xi\in\cT$ and $C\in\cM_j$
there exist $U_C,V_C\in\Fq[T]$ with $\gcd(U_C,V_C)=1$, $|U_C|<|V_C|\leq Q$ and
\begin{equation}\label{eq:t1-1}
\left|C\xi-\frac{U_C}{V_C}\right|< \frac{1}{|V_C|Q}.
\end{equation}

Let $\cT_1$ denote the set of such $\xi \in \cT$ that
for all $C\in \cM_j$ 
\begin{equation}\label{eq:Vc_large}
|V_C|\geq q^{\omega}. 
\end{equation}

Define $\cT' = \cT\setminus \cT_1$, that is, the set of those $\xi\in\cT$ that there exist some polynomial $\bar{C}\in\cM_j$ such that there are no polynomials $U_{\bar{C}}, V_{\bar{C}}$ satisfying \eqref{eq:t1-1} and \eqref{eq:Vc_large} with $\bar{C}$ in place of $C$.

By Lemma~\ref{lemma:dio}, there exist $\bar{U}, \bar{V}\in\Fq[T]$ such that 
$$
\left|\bar{C}\xi -\frac{\bar{U}}{\bar{V}} \right|<\frac{1}{|\bar{V}|Q}, \quad \gcd(\bar{U},\bar{V})=1,  |\bar{V}|<Q.
$$
If $\xi\in\cT'$, we also have $|\bar{V}|<q^{\omega}$. Write $U/V=\bar{U}/\bar{C}\bar{V}$ with coprime polynomials $U,V$ and $\gamma=\xi-U/V$. Then we have
\begin{equation}\label{eq:V-gamma}
|V|\leq q^j |\bar{V}| <q^{\omega + j}\mand |\gamma|<\frac{1}{q^j|\bar{V}|Q}.
\end{equation}
Moreover, as $|\xi|\geq q^{-k}$, we also have $|V|>1$. Indeed, if $V=1$, then $\gamma=\xi$ contradicting to \eqref{eq:k}, \eqref{eq:V-gamma} and the definition of $Q$.

Define the following subsets of $\cT'$,
\begin{align*}
    \cT_2&=\{\xi\in \cT': 1<|V|\leq q^{j}, |\gamma|<q^{-n} \},\\
    \cT_3&=\{\xi\in \cT': q^{j} <|V|\leq q^{j+\omega }, |\gamma|<q^{-n} \},\\
    \cT_4&=\{\xi\in \cT': 1<|V|\leq q^{j+\omega}, |\gamma|\geq q^{-n} \}.
\end{align*}
We estimate $f_\cS$ separately in each subsets $\cT_1,\dots, \cT_4$.

For $\xi\in \cT_1$, we have by Lemma~\ref{lemma:irred_sum}, that
\begin{align*}
f_\cS(\xi)&\ll q^j n^{3/2} \left(q^{n-j-\omega/2} + q^{4(n-j)/5}+q^{\omega/2+(n-j+1)/2}\right)\\
&\ll n^{3/2} \left(q^{n-\omega/2}+q^{(4n+j)/5} +q^{(n+j+\omega+1)/2}\right) \ll \frac{q^{n-j}}{n-j}. 
\end{align*}

For $\xi\in\cT_2$, we have $\deg V\leq j < n-j$ thus for all $P\in\cI_{n-j}$
$$
\left\|\frac{PU}{V} \right\|\geq \frac{1}{|V|}\geq q^{-j}
$$
as both $U$ and $P$ are coprime to $V$. Moreover,
$$
\left\|P\gamma \right\|= |P||\gamma|< q^{-j}
$$
and thus
$\|P\xi\|\geq q^{-j}$. Then we have by \eqref{eq:sum}, that
$$
f_\cS(\xi)= \sum_{P\in\cI_{n-j}} \sum_{C\in\cM_j}   \e (CP\xi)=0.
$$

If $\xi\in\cT_3$, then $|\gamma|<q^{-n}$ and for $P\in\cI_{n-j}$ we have
$$
\|P\gamma\|<q^{-j} \mand (P\gamma)_{-j-1}=\gamma_{-n-1}.
$$
Whence, by \eqref{eq:sum},
\begin{equation}\label{eq:splitting-2}
    \sum_{C\in \cM_{j}}\e(CP\xi)=
    \left\{ 
    \begin{array}{cl}
         q^j \psi(\gamma_{-n-1})\e\left(T^jP\frac{U}{V}\right) & \text{if } \left\|P\frac{U}{V} \right\|<q^{-j} , \\
         0 & \text{otherwise.} 
    \end{array}
    \right.
\end{equation}

For $z\in\Fq$, let $\cW_z$ be the set of such residues $W\in\Fq[t]$ modulo $V$ that, 
$$
\|WU/V \|<q^{-j}, \quad  (WU/V)_{-j-1}=z \mand \gcd(W,V)=1.
$$
By Lemma~\ref{lemma:combined_PNT-2}, we have
$$
\#\{P\in \cI_{n-j}: P \bmod V \in \cW_z\}=\frac{\#\cW_z}{\Phi(V)}\frac{q^{n-j}}{n-j} + O\left(\#\cW_z \frac{\deg Vq^{(n-j)/2}}{n-j} \frac{|V|}{\Phi(V)}\right).
$$
Moreover, replacing $W$ by $cW$ for some $c\in\Fq^*$, we have $\#\cW_{z}=\#\cW_{cz}$. Also, we have the trivial bound
$$
\#\cW_{z}\leq \frac{|V|}{q^{j+1}}.
$$

Thus by \eqref{eq:splitting-2} we have
\begin{align*}
    \sum_{P\in\cI_{n-j}}\sum_{C\in\cM_{j}}\e(CP\xi)
    &=q^j \psi(\gamma_{-n-1}) \sum_{z\in\Fq} \psi(z) \#\{P\in \cI_{n-j}: P \bmod V \in \cW_z\}\\
    & =\frac{q^{n}}{n-j}  \psi(\gamma_{-n-1}) \sum_{z\in\Fq} \psi(z) \frac{\#\cW_z}{\Phi(V)}
    +
    O\left( \frac{ \deg V q^{(n-j)/2 }}{n-j} \frac{|V|^2}{\Phi(V)}\right)\\
    & \ll \frac{q^{n}}{n-j}   \frac{|\#\cW_1- \#\cW_0|}{\Phi(V)}
    +
    \frac{ (j+\omega) q^{(n-j)/2 }}{n-j} \frac{|V|^2}{\Phi(V)}\\
    & \leq \frac{q^{n}}{n-j}   \frac{|V|}{q^{j+1}\Phi(V)}
    +
     \frac{ (j+\omega)
    q^{(n-j)/2 }}{n-j}\frac{|V|^2}{\Phi(V)}\\
    & \ll \frac{q^{n-j-1}}{n-j} (1+\log 3j)  
    +
     \frac{ (j+\omega) \log_q(j+\omega) q^{(n-j)/2 +3j }}{n-j}\\
     &\ll \frac{\log j}{n-j}q^{n-j}
\end{align*}
by \eqref{eq:phi}.

Assume, that $\xi\in\cT_4$. As before, 
we have by \eqref{eq:sum}, that
\begin{equation}\label{eq:starting_sum}
\sum_{B\in\cM_j}   \e (BP\xi)=
\left\{
\begin{array}{cc}
    q^j \e (T^j P\xi) & \text{if } \|P\xi\|<q^{-j},  \\
    0 & \text{otherwise}. 
\end{array}
\right.
\end{equation}

We can write
$$
(T^jP\xi)_{-1}=(P\xi)_{-j-1}=\left(\frac{PU}{V}\right)_{-j-1}+\left(P\gamma\right)_{-j-1}.
$$
Put
$$
q^\ell=|\gamma|q^{n+1}.
$$
By \eqref{eq:Q}, \eqref{eq:V-gamma} and the definition of $\cT_4$ we have
\begin{equation}\label{eq:ell}
    1\leq \ell < \frac{3j}{2}-\deg V+1<n-j.
\end{equation}

Then, the value $(T^jP\xi)_{-1}$ depends only on the residue class of $P$ modulo $\cR_{\ell, V}$. Indeed, if $P\equiv A \mod \cR_{\ell,V}$ with $\deg A = n-j$, then 
$$
(PU/V)_{-j-1}=(AU/V)_{-j-1}
$$ 
as $P\equiv A \pmod V$ and 
$$
(P\gamma)_{-j-1}-(A\gamma)_{-j-1}=((P-A)\gamma)_{-j-1}=0
$$
as $|P-A|< q^{n-j-\ell}$.

By Lemma~\ref{lemma:combined_PNT-2}, we have for $A\in G_{\ell,V}$ that
$$
\#\left\{P\in \cI_{n-j}: P\equiv A \mod \cR_{\ell,V}\right\}
= \frac{1}{q^\ell\Phi(V)}\frac{q^{n-j}}{n-j} +O\left((\ell + \deg V)\frac{q^{(n-j)/2}}{n-j}\frac{|V|}{\Phi(V)}\right) .
$$

Put $s=n-j-\deg V-\ell-1$. By \eqref{eq:ell}, $s\geq 0$
and  write each $A\in G_{\ell, V}$ as
$$
A=A_0VT^{s}+A_1 \quad \text{with } A_0\in\cM_\ell, \ \deg A_1< \deg V,\ \gcd(A_1,V)=1.
$$

By \eqref{eq:CRT}, $A_0$ takes all polynomial in $\cM_\ell$ in the same frequency as $A$ runs in $G_{\ell,V}$, thus by \eqref{eq:sum} we have
\begin{equation}\label{eq:gammaC_0}
\sum_{A\in G_{\ell,V}}\e(T^jA\xi)=\sum_{\substack{A_1:\\ \gcd(A_1,V)=1}} \e(T^j A_1 U/V) \sum_{A_0\in \cM_\ell} \e(VT^{j+s}V A_0\gamma)=0
\end{equation}
as 
$$
\|T^{j+s}V\gamma\|=\|T^{n-\deg V-\ell-1}V\gamma\|=q^{-1}.
$$

Using \eqref{eq:starting_sum} and collecting the irreducible polynomials with respect to the reduced residue classes modulo $\cR_{\ell,V}$, we get by \eqref{eq:gammaC_0} that
\begin{align*}
f_\cS(\xi)&= \sum_{P\in\cI_{n-j}}\ \sum_{C\in\cI_j}   \e (CP\xi)\\
& = q^j
\sum_{A \in G_{\ell, V}} 
\sum_{\substack{P\equiv A \bmod \cR_{\ell,V}}}\e(T^jP\xi)\\
& = q^j \sum_{z\in \Fq} \psi(z)
\sum_{\substack{A\in G_{\ell, V}\\ (T^jA\xi)_{-1}=z  }}  \# \left\{P\in \cI_{n-j}: P\equiv A \mod \cR_{\ell ,V}\right\}
\\
& 
= q^j \sum_{z\in \Fq} \psi(z)\sum_{\substack{A\in G_{\ell, V}\\ (T^jA\xi)_{-1}=z  }}  \Bigg(\frac{1}{ \Phi(V)}\frac{q^{n-j-\ell}}{n-j}
+O\left((\ell +\deg V)\frac{q^{ (n-j)/2}}{n-j}\frac{|V|}{\Phi(V)}\right)\Bigg)\\
& 
= 
\frac{q^{n-\ell}}{n-j}\frac{1}{\Phi(V)} \sum_{A\in G_{\ell,V}}\e(T^jA\xi)
 +O\left((\ell +\deg V)|V|\frac{q^{ (n+j)/2+\ell}}{n-j}\right)
\\
& 
\ll
(\ell +\deg V)|V|\frac{q^{ (n+j)/2+\ell}}{n-j}
\\
&\ll n  \frac{q^{n+1}}{Q} \frac{q^{(n+j)/2}}{n-j}
\ll \frac{q^{n-j}}{n-j}.
\end{align*}

\end{proof}

\section{Proof of Theorem~\ref{thm:main}}\label{sec:proof}
Put
$$
k= n- j.
$$

Define
\begin{equation}
\begin{split}\label{eq:I_split}
I_1&=\int_{|\xi|<q^{-k}}f_\cA(\xi)f_\cB(\xi)f_\cS(-\xi)\mathrm{d}\xi,
 \\
I_2&=I-I_1=\int_{q^{-k}\leq |\xi|<1}f_\cA(\xi)f_\cB(\xi)f_\cS(-\xi)\mathrm{d}\xi.
\end{split}
\end{equation}

By \eqref{eq:AB_bound}, we have
\begin{equation*}
    1\leq j<n/7-4\log_q n-3/7
\end{equation*}
if $n$ is large enough in terms of $\varepsilon$. Then we have by Lemma~\ref{lemma:main}, that 
\begin{equation}\label{eq:I2}
\begin{split}
I_2
&\ll \frac{\log j}{n}q^{n-j}   \int_{q^{-k}\leq |\xi|<1} |f_\cA(\xi)||f_\cB(\xi)|\mathrm{d}\xi\\
 &   \ll \frac{\log j}{n}q^{n-j}  \int_{\mathbf{T}} |f_\cA(\xi)||f_\cB(\xi)|\mathrm{d}\xi\\
 &   \ll \frac{\log j}{n}q^{n-j}
 \left(
 \int_{\mathbf{T}} |f_\cA(\xi)|^2\mathrm{d}\xi\int_{\mathbf{T}} |f_\cB(\xi)|^2\mathrm{d}\xi\right)^{1/2}\\
  &   \ll \frac{\log j}{n}q^{n-j}
 \left( \# \cA \#\cB \right)^{1/2}.
\end{split}
\end{equation}

Write
$$
f_{\cU}(\xi)=\sum_{\deg B <k}\e(B\xi),
$$
where the sum is over all (not just monic) polynomials.
Then
\begin{equation*}
f_{\cU}(\xi)=\sum_{\tilde{B} \in \cM_k}\e((\tilde{B}-T^k )\xi)=\e(-T^k \xi) \sum_{\tilde{B} \in \cM_k}\e(\tilde{B}\xi) =
\left\{
\begin{array}{cl}
q^{k}      & \text{if }  |\xi|<q^{-k},\\
0     & \text{otherwise,}
\end{array}
\right.
\end{equation*}
by \eqref{eq:sum}.

Then we have
\begin{equation}\label{eq:I1}
\begin{split}
I_1&=\int_{|\xi|<q^{-k}}f_\cA(\xi)f_\cB(\xi)f_\cS(-\xi)\mathrm{d}\xi\\
&=\frac{1}{q^{k}}\int_{|\xi|<q^{-k}}f_\cA(\xi)f_\cB(\xi)f_\cS(-\xi)f_{\cU}(-\xi)\mathrm{d}\xi\\
&=\frac{1}{q^{k}}\int_{\mathbf{T}}f_\cA(\xi)f_\cB(\xi)f_\cS(-\xi)f_{\cU}(-\xi)\mathrm{d}\xi\\
&=\frac{1}{q^{k}}\int_{\mathbf{T}}\left(\sum_{G\in\cM_n}u_G\e(G\xi)\sum_{H\in\cM_n}w_H\e(-H\xi) \right)\mathrm{d}\xi\\
&=\frac{1}{q^{k}}\sum_{G\in\cM_n}u_G w_G,
\end{split}
\end{equation}
where $u_G$ is defined as \eqref{eq:u_def}
and
$$
w_H=\#\{(C,P): H\equiv CP \mod \cR_{n-k},  |C|=q^j, P\in\cI_{n-j}\}.
$$

As $j=n-k$, and $C$ is monic polynomial of degree $j$, it is invertible modulo $\cR_{n-k}$.

Then for a fixed $C$, the number of solution $P\in \cI_{n-j}$ to
$$
H\cdot C^{-1} \equiv P \mod \cR_{n-k}
$$
is
$$
\frac{1}{q^{n-k}}\frac{q^{n-j}}{n-j} + O\left(\frac{n-k}{n} q^{(n-j)/2} \right)\gg \frac{1}{q^{n-k}}\frac{q^{n-j}}{n-j}
$$
by Lemma~\ref{lemma:combined_PNT-2}.

Moreover, all  $q^j=q^{n-k}$ polynomial $C$ are incongruent modulo $\cR_{n-k}$, thus we have
$$
w_H\gg \frac{q^{n-j}}{n-j}.
$$
Then, it follows from \eqref{eq:I1}, that
\begin{equation}\label{eq:I1_fi}
    I_1\geq c_1 \frac{1}{n} \sum_{G\in\cM_n} u_G =c_1 \frac{\#\cA \#\cB }{n}
\end{equation}
for some absolute constant $c_1>0.$

We have
\begin{align*}
\frac{q^{j}}{\log j}\geq  \rho \frac{e^{\eta} \log_q\log_q \rho }{\log 2 +\log_q\log_q \rho  +\log_q\log q+ \log \eta} \geq \rho e^{\eta /2} 
\end{align*}
if $\eta$ is large enough. Thus combining \eqref{eq:I_split}, \eqref{eq:I2} and \eqref{eq:I1_fi} we get
\begin{align*}
I=I_1+I_2&\geq c_1\frac{\#\cA \#\cB}{n} -c_2\frac{\log j}{n}q^{n-j}
 \left( \# \cA \#\cB \right)^{1/2} \\
 &=c_1\frac{\#\cA \#\cB}{n}
 \left( 1 - c_3\log j \frac{q^{n-j}}{
 \left( \# \cA \#\cB \right)^{1/2}} \right)
 \\
 &\geq c_1\frac{\#\cA \#\cB}{n}
 \left( 1 -c_3 \rho \, \frac{\log j}{
 q^{j}}
 \right)\\
 &\geq c_1\frac{\#\cA \#\cB}{n}
 \left( 1 -c_3 e^{-\eta/2}
 \right)
 .
\end{align*}
If $\eta$ is large enough, we get
 \eqref{eq:goal} and hence Theorem~\ref{thm:main}.

\section*{Acknowledgment}
The author was supported by the Austrian Science Fund FWF grant  P~31762.
He wishes to thank to the anonymous referees for the careful reading and their very helpful reports.

\end{document}